\documentclass[11pt]{amsart}
\usepackage[margin=0.78in]{geometry}

\usepackage[utf8]{inputenc}
\usepackage{amsmath,amssymb,amsfonts,amsthm}
\usepackage{hyperref}
\usepackage{xcolor}
\usepackage{tikz}

\newtheorem{theorem}{Theorem}[section]
\newtheorem*{theorem*}{Theorem}
\newtheorem{corollary}[theorem]{Corollary}

\newtheorem{proposition}[theorem]{Proposition}
\newtheorem{lemma}[theorem]{Lemma}

\theoremstyle{definition}
\newtheorem{definition}[theorem]{Definition}
\newtheorem{remark}[theorem]{Remark}

\newtheorem{assumption}{Assumption}

\newcommand{\bR}{\mathbb{R}}
\newcommand{\bN}{\mathbb{N}}
\newcommand{\cH}{\mathcal{H}}
\newcommand{\cL}{\mathcal{L}}

\newcommand{\cM}{\mathcal{M}}

\newcommand{\cD}{\mathcal{D}}
\newcommand{\cF}{\mathcal{F}}
\newcommand{\hd}{\tilde{d}}

\newcommand{\pd}[2]{\frac{\partial#1}{\partial #2}}

\newcommand{\supp}[1]{\mathrm{supp}\left(#1\right)}

\newcommand{\norm}[1]{\left\lvert{#1}\right\rvert}

\newcommand{\lsf}{\mathrm{LSF}}

\title[Non-uniqueness in MCF]{Non-uniqueness in Mean Curvature Flow: \\Non-canonical solutions via the parabolic Allen--Cahn}
\author{J.~M.~Daniels-Holgate}
\address{School of Mathematical Sciences Queen Mary University of London Mile End Road London E1 4NS UK}
\email{j.danielsholgate@qmul.ac.uk} 

\begin{document}
\begin{abstract}When mean curvature flow evolves non-uniquely, the flow is said to fatten. The work of Ilmanen shows that any weak MCF is supported inside the fattening, and work of Hershkovits--White identified canonical weak flows supported on the boundary of the fattening, known as the outermost flows. It is natural to ask, when the flow fattens, are there weak mean curvature flows supported strictly inside the fattening? Outside of some special cases (e.g. flow from cones), this question was entirely open.
We show these interior flows exist, providing a general construction for non-outermost flows as limits of solutions to the parabolic $\varepsilon$-Allen--Cahn. This gives the first examples of closed, non-trivial, non-canonical, integral Brakke motions. 
As part of this construction, we study the $\varepsilon$-Allen--Cahn flow from low regularity initial data, and our results demonstrate the existence of integral Brakke motions from fractal sets. 
This includes the existence portion of Hershkovits's work on mean curvature flow from Reifenberg sets.
\end{abstract}
\maketitle
\section{Introduction}
    Given a smooth double-well potential $F$, the parabolic $\varepsilon$-Allen--Cahn ($\varepsilon$-AC) equation is given by 
        \begin{align}
        \pd{}{t} u^\varepsilon &= \Delta u^\varepsilon - \frac1{\varepsilon^2}f(u^\varepsilon)
        \end{align}
        where $f=F'$. The $\varepsilon$-AC equation was proposed by \cite{AllenCahn} a model of phase transition in metallic alloys, where the sharp interface is replaced by a phase field.

        A family of hypersurfaces $\{M_t\}_{t\in[0,T]}$ is said to move by mean curvature flow if 
        \begin{align*}
            (\partial_t \mathbf{x})^\perp = {\mathbf{H}}(\mathbf{x},t),
        \end{align*}
        where ${\mathbf{H}}=- H\nu$ denotes the mean curvature vector of $M_t$ at $\mathbf{x}$. Whilst the mean curvature flow from smooth hypersurfaces is known to remain unique as long as it is smooth, it has long been known that flow from immersed hypersurfaces can have multiple evolutions. More recently, work of Ilmanen--White \cite{IW} has shown mean curvature flow from smooth, closed, embedded hypersurfaces can also evolve non-uniquely after the first singular time.

        We will explore the relationship between the solutions to the $\varepsilon$-Allen--Cahn equation and mean curvature flow in the presence of non-unique evolution.
        
        \subsection{Background}
        The convergence of the nodal set of solutions to \ref{eAC} to smooth motions by mean curvature flows was established by De Mottoni--Schatzman, \cite{MS1}, via asymptotic expansion, and later Chen,\cite{Chen}, via a maximum principle argument. 
        
        Mean curvature flow (MCF) of compact hypersurfaces is known to form singularities in finite time, and there are two key notions of weak mean curvature flow through these singularities: The Level-Set flow (LSF), and the Brakke Flow. See Definitions \ref{brakkeflow} and \ref{levelsetflow}. The level-set flow was introduced independently by Evan--Spruck and Chen--Giga--Goto, and can be started from merely continuous initial data.
         A key feature of the LSF is that it captures non-unique evolution of mean curvature flow by developing an interior, this phenomenon is known as \textit{fattening}, (Definition \ref{fattening}). Whilst Evans--Spruck were able to provide an example of fattening LSF with immersed initial data (the classical figure 8) in their original papers \cite{ES1}, 
          it was not until the work of Ilmanen--White \cite{WICM, IW} that it was known if the LSF starting from a smooth compact hypersurface could fatten. See also the authors work with Chodosh and Schulze \cite{CDHS}, where a fattening dichotomy is established for compact MFC encountering isolated conical singularities, and also the works of Lee--Zhao, \cite{LeeZhao}, which, when combined with the work of Chodosh--Daniels-Holgate--Schulze yield the existence of smooth compact initial data whose level-set flow fattens at the first singular time.

         Fattening presents a major obstacle for using the MCF as a tool for solving topological problems: Whilst one might be able to say something about the outermost flows (see Definition \ref{outermost}), as in the canonical neighbourhood theorem of \cite{CDHS}, if one is trying to flow a \textit{family} of hypersurfaces, fattening represents a discontinuous jump in the family of flows, across which topological understanding is lost. This motivates understanding flows supported within the fattening.

         Very little was know about the composition of the LSF, and this work is the first key step towards developing a full understanding. Work of Ilmanen \cite{IlmanenER} has shown that any Brakke flow must be supported within the LSF. The converse, that there are flows within the fattening, was considered heuristically true within the field, though without formal proof which we now provide. Indeed, there has been a lot of evidence for the existence of interior flows. Work of Chen--Bernstein--Wang, \cite{CBW}, constructs non-outermost flows from cones, formally, heteroclinic flows between stable expanders. In the case of curve shortening flow, work of Kim--Tonegawa, \cite{KimTonegawa} suggests that two triple junctions are a possible a resolution from the immersed cross in curve shortening flow. Note, the Kim--Tonegawa triple junction resolution must be interior, as the outermost flows from the cross are smooth. See Section \ref{thefigure8}.

        Returning to the discussion of the $\varepsilon$-Allen--Cahn flow, in the absence of fattening, i.e. $\cH^{n+1}(u^{-1}({0}))=0$, Evans--Soner--Souganidis \cite{ESS} demonstrated the convergence of the phase-field $u^\varepsilon$ to a solution $u$ of the level-set flow, hence establishing the $\varepsilon$-Allen--Cahn flow as a model of weak mean curvature flow through singularities. 
        They make no claims regarding the case in which the LSF fattens. 
        
        In \cite{Ilmanen}, Ilmanen took an entirely different approach to relating the $\varepsilon$-Allen--Cahn flow to weak MCF. Using methods from geometric measure theory, Ilmanen demonstrated that, for a certain class of solutions, the 1-parameter family of energy measures
        \begin{align*}
            \left\{d \mu^\varepsilon(dx,t) \right\}_{t\in[0,\infty)}= \left\{\left[ \frac{\varepsilon\norm{\nabla u^\varepsilon}^2(x,t)}{2} +\frac{F(u^\varepsilon(x,t))}{\varepsilon} \right]dx\right\}_{t\in[0,\infty)}
        \end{align*}
        converges to a Brakke flow (a 1-parameter family of Radon measures) as the parameter $\varepsilon\to0$. One notes that this convergence is as varifolds, not only as Radon measures. Ilmanen achieved this by establishing $\varepsilon$-versions of Brakke's inequality and Huisken's monotonicity formula, under the assumption that the discrepancy
        measure, 
        \begin{align*}
            d\xi^\varepsilon=\left[\frac{\varepsilon\norm{\nabla u^\varepsilon}^2(x,t)}{2} -\frac{F(u^\varepsilon(x,t))}{\varepsilon} \right] dx
        \end{align*}
        is non-positive. Ilmanen verified this was a practical assumption by giving an explicit construction of initial data for the $\varepsilon$-AC equation with non-positive discrepancy from boundaries of sets with upper $\cH^n$ density bounds that could be approximated by smooth hypersurfaces. Soner later extended the class of valid initial data by showing rapid decay of the discrepancy as $\varepsilon\to0$, \cite{Soner1}, though this work requires the hypersurface be $C^2$. Finally, Tonegawa demonstrated that the limiting varifolds have integer multiplicity, cementing the notion of $\varepsilon$-Allen--Cahn flow as a suitable a model of weak MCF.

    \subsection{Results}
    A key feature of Ilmanen's framework is that, unlike Evans--Soner--Souganidis \cite{ESS}, no assumption is made on the level-set flow of the hypersurface. As such, the support of the (subsequential) limiting flow is ambiguous, the only restriction being 
            \begin{align*}
                \supp{\cM^\infty}\subset \lsf(M).
            \end{align*} 
    Indeed, this leaves open the possibility that the Brakke flow synthesised from Ilmanen's process is supported entirely on the interior. We demonstrate this is possible.
    
    \begin{theorem}\label{main}
        Let $M_0$ be a compact (generalised) hypersurface. Suppose that $\lsf(M_0)$ fattens in finite time. Then, for each $X_0\in\lsf(M_0)$, there exists a unit-regular, integral Brakke flow $\{\mu_t\}_{t\in[0,\infty)}$ starting from $M_0$, supported at $X_0$.
    \end{theorem}
    
    \begin{remark}
    Note, for generalised hypersurfaces, `starting from' should be taken to mean that the Hausdorff distance of $\supp{\mu_t}$ to the initial data goes to $0$ as $t\to0$. This is strictly weaker than the usual notion of starting from. When the initial data is smooth, however, a short argument (Theorem \ref{goodstartingfrom} and Remark \ref{goodstartingfromremark}) shows that we can use the classical notion of `starting from'.
    \end{remark}
   Applications of this theorem are givien in Section \ref{examples}, where we study the $g$-Wagon-Wheel, which was shown by Ilmanen--White, \cite{IW}, to fatten in finite time, and also the flow from the figure 8.

    The core result of this article includes existence for integral Brakke flows starting from very weak initial conditions:
    \begin{corollary}
        Let $\Omega\subset\bR^{n+1}$ be a compact set such that its boundary $M=\partial\Omega$ satisfies 
        \begin{align*}
            \cH^{n+\kappa}(M)<\infty
        \end{align*}
        for some $\kappa\in[0,1)$ and there exists $\delta,\theta>0$ such that
        \begin{align*}
            \frac{\cH^{n+\kappa}(M\cap B^{n+1}(p,\gamma))}{\gamma^{n+\kappa}}\geq \theta
        \end{align*}
    for all $p \in M$ and $\gamma\in(0,\delta]$.
    Then, there exists an integral Brakke flow starting from $M$.
    \end{corollary}
    This existence result follows from showing when starting the $\varepsilon$-Allen--Cahn flow from an initial condition derived from such generalised hypersurfaces, the $\varepsilon$-Allen--Cahn energy is instantaneously finite, and thus we can appeal to the results of Ilmanen, Soner and Tonegawa, see Theorem \ref{bflimit}.

    This low-regularity initial data includes Reifenberg sets, and so this Corollary reproduces the existence result of Hershkovits, \cite{Her_reif}. Note, Hershkovits also establishes uniqueness and short-time regularity of such flows, which we are unable to achieve using the methods outline in this article.

        When $\cH^n(M)<\infty$, Theorem \ref{main} is novel for $X_0 \notin \partial\lsf(M_0)$, as, given an initial hypersurface with fattening level-set flow Hershkovits--White, \cite{hershwhite} introduced the \textit{inner} and \textit{outer} flows, and construct (unit-regular) integral Brakke flows (when $\cH^n(M)<\infty$) supported on these weak-set flows via Ilmanen's elliptic regularisation process, \cite{IlmanenER}.

    Our process is very simple. We assume $M$ is a generalised hypersurface such that $\lsf(M)$ fattens in finite time. We foliate a neighbourhood of $M$ and demonstrate we can start the $\varepsilon$-Allen--Cahn flow from $2\chi_{M_s}-1$ for each leaf $M_s$ of the foliation, yielding a unique solution $u^\varepsilon_s$. This removes any bias introduced when synthesising a suitable initial conditions. Our results are not dependent on our choice of initial condition, one merely needs to pick initial data for the hypersurface that is not dependent on approximating processes. 
    We then show for each $\varepsilon>0$ and any $X\in \lsf(M_0)$, there is an $s$ such that $X$ is in the nodal set of $u^\varepsilon_s$. The desired Brakke flow is then constructed as a diagonal limit in $s$ and $\varepsilon$. The key ingredient for our results is the maximum principle of de Mottoni--Schatzman.

    We note that such a result was thought unlikely; typically one constructs initial data for the $\varepsilon$-Allen--Cahn flow by smoothing the initial hypersurface to at least $C^2$. Indeed, for rough initial data (and thus entertaining the possibility of instantaneous fattening), the process discussed by Ilmanen \cite[§1.4]{Ilmanen} and Modica \cite[§2(2)]{Modica} smooths the hypersurface to one side. Comparing this process to the construction of the outermost flows by Hershkovits--White, \cite{hershwhite}, it is reasonable to expect such a process would naturally select for the outermost flows.
    
    \subsection{Acknowledgments}
        The author is grateful to Otis Chodosh and Huy The Nguyen for their helpful discussion.
        The author acknowledges support from the EPSRC through the grant EP/Y017862/1, Geometric Flows and the Dynamics of Phase Transitions.
\section{Background and Preliminary results}
\begin{definition}
    Let $\Omega\subset \bR^{n+1}$ be a bounded open set set. We call $M = \partial \Omega$ a \textit{Generalised Hypersurface} if, for some $\kappa\in[0,1) $
    \begin{align*}
        \cH^{n+\kappa}(M)<\infty
    \end{align*} 
    and $M$ has a lower density bound: i.e. there exists $\delta,\theta>0$ such that
                \begin{align*}
                    \frac{\cH^{n+\kappa}(M\cap B^{n+1}(p,\gamma))}{\gamma^{n+\kappa}}\geq \theta
                \end{align*}
            for all $p \in M$ and $\gamma\in(0,\delta]$.
\end{definition}

\subsection{Weak Mean Curvature Flow}
\begin{definition}[Integral Brakke Flow \cite{Brakke,IlmanenER}]\label{brakkeflow}
    We follow the formalism of \cite{white21}. 
    An ($n$-dimensional) \textit{integral Brakke flow} in $\mathbb{R}^{n+1}$ is a 1-parameter family of Radon measures $\{\mu_t\}_{t\in I}$ over an interval $I\subset \mathbb{R}$ such that:
    \begin{enumerate}
       \item For almost every $t$ there exists and integral $n$-dimensional varifold $V(t)$ with $\mu_t=\mu_{V(t)}$ so that $V(t)$ has locally bounded first variation and has mean curvature $\mathbf{H}$ orthogonal to $\mathrm{Tan}(V(t),\cdot)$ almost everywhere.
        \item For a bounded interval $[t_1,t_2]\subset I$ and any compact set $K$ \begin{align*}
            \int^{t_2}_{t_1}\int_K(1+|\mathbf{H}|^2)\,d\mu_t\,dt<\infty\, .
        \end{align*}
        \item If $[t_1,t_2]\subset I$ and $f\in C^1_c(\mathbb{R}^{n+1}\times [t_1,t_2])$ has $f\geq0$ then \begin{align}\label{Brakkeinequality}
            \int f(\cdot,t_2)\,d\mu_{t_2}-\int &f(\cdot,t_1)\,d\mu_{t_1}\nonumber\\
            \leq \int^{t_2}_{t_1}\int_K& \Big(-|\mathbf{H}|^2f+\mathbf{H}\cdot\nabla f +\frac{\partial}{\partial t}f\Big)\,d\mu_t\, dt
        \end{align}
    \end{enumerate}
    We write $\mathcal{M}$ for a Brakke flow $\{\mu_t\}_{t\in I}$ to refer to the family of measures $I\ni t\mapsto \mu_t $ satisfying Brakke's inequality (\ref{Brakkeinequality}).
\end{definition}

\begin{definition}[Gaussian Density]
    The Guassian density at a point $X_0=(\mathbf{x}_0,t_0)$ and scale $r\in(0,\infty)$ is defined as
    \begin{align*}
    \Theta_\mathcal{M}(X_0,r):=\int_{\mathbb{R}^{n+1}}\rho_{X_0}(\mathbf{x},t_0-r^2)\, d \mu_{t_0-r^2}
    \end{align*}
\end{definition}

\begin{definition}[Unit-regular \cite{White09}] 
    
An integral Brakke flow $\mathcal{M}=\{\mu_t\}_{t\in I}$ is said to be \textit{unit-regular} if $\mathcal{M}$ is smooth in some space-time neighbourhood of any spacetime point for $X$ for which $\Theta_\mathcal{M}(X)=1$.
\end{definition}

\begin{definition}[Weak and Level set flow, \cite{IlmanenER, hershwhite}]\label{levelsetflow}
    Let $K\subset\mathbb{R}^{n+1}$ be closed. A one-parameter family of closed sets, $\{K_t\}_{t\geq0}$, with initial condition $K_0=K$ is said to be a \textit{weak set flow} for $K$ if for every smooth mean curvature flow $M_t$ of compact hypersurfaces defined on $[t_0,t_1]$, we have
    \begin{align*}
        K_{t'}\cap M_{t'}=\emptyset \implies K_t\cap M_t=\emptyset
    \end{align*}
    for all $t'\in[t_0,t_1]$ and  $t\in [t',t_1]$.
    
    The \textit{level set flow} is defined as the maximal weak set flow, i.e.~the union of all weak set flows from $K$. We also write $F_t(K)=K_t$ to denote the $t$ time-slice of the level set flow of the set $K$.
    \end{definition}
    \begin{definition}\label{fattening}
        Let $M$ be a compact, smoothly embedded hypersurface. The \textit{fattening time} of the level set flow of $M$ is defined as
        \begin{align*}
            T_\mathrm{fat}:=\inf \{t>0: F_t(M)\ \textrm{has non-empty interior}\}.
        \end{align*}
    \end{definition}
    
    \begin{definition}[Outermost flows, \cite{hershwhite}]\label{outermost}
        Let $M_0^n\subset \bR^{n+1}$ be a smooth, compact hypersurface, and let $U$ be the compact region bounded by $M_0$. Let $U'=\overline{U^c}$. Using the set theoretic formulation of the level set flow, we define space-time tracks of the evolution of $U, U'$ under level-set flow by
        \begin{align*}
            \mathcal{U}:=\{(\mathbf{x},t)\subset \bR^{n+1,1}|\ \mathbf{x}\in F_t(U)\},\\
            \mathcal{U}':=\{(\mathbf{x},t)\subset \bR^{n+1,1}|\ \mathbf{x}\in F_t(U')\}.
        \end{align*}
        We hence define
        \begin{align*}
            M(t):=\{\mathbf{x}\in \bR^{n+1}|\ (\mathbf{x},t)\in \partial \mathcal{U}\},\\
            M'(t):=\{\mathbf{x}\in \bR^{n+1}|\ (\mathbf{x},t)\in \partial \mathcal{U}'\}.
        \end{align*}
        We call $t\mapsto M(t)$ the \textit{outer flow} from $M_0$ and $t\mapsto M'(t)$ the \textit{inner flow} from $M_0$. One can see $M(t), M'(t)\subset F_t(M)$. The elliptic regularisation process of Ilmanen, \cite{IlmanenER}, shows there exist unit-regular, integral Brakke flows supported on these sets. In a slight abuse of terminology, we also refer to these Brakke flows as the inner and outer flows.
    \end{definition}
\subsection{Parabolic Allen--Cahn}
We fix the potential $F=\frac1{2}(1-u^2)^2$ and hence $f=2u(u^2-1)$.

The initial value problem for the parabolic $\varepsilon$-Allen--Cahn is
\begin{align}
    \pd{}{t} u^\varepsilon &= \Delta u^\varepsilon - \frac1{\varepsilon^2}f(u^\varepsilon)\label{eAC},\\
    u^\varepsilon(\cdot,t)&=u_0(\cdot)\label{initdata}.
    \end{align}

Recall the key existence result of De Mottoni--Schatzman.
\begin{theorem}[Existence,\text{\cite[Theorem 1.2]{MS1}}]\label{exist}
    For any $u_0\in L^\infty(\bR^{n+1})$ and any $\varepsilon>0$, there exists a unique function $u\in L^\infty(\bR^{n+1}\times[0,T])$ for all $T>0$, continuous from $\bR^+$ to $\cD'(\bR^{n+1})$, that satisfies \ref{eAC} in the sense of distributions on $\bR^{n+1}\times(0,\infty)$ and \ref{initdata}. 
    Moreover, such a function is infinitely differentiable over $\bR^{n+1}\times(0,\infty)$, and
        \begin{align*}
            \norm{u(\cdot,t)}_\infty \leq \max(\norm{u_0}_\infty,1), \textrm{ for all values of } (x,t)\in\bR^{n+1}\times(0,\infty)
        \end{align*}
\end{theorem}
The proof uses Duhamel's principle and a contraction mapping argument. As such, one can immediately deduce regularity and the expected derivative estimates for parabolic equations.
\begin{corollary}\label{derivative}
    Let $u^\varepsilon$ be the unique solution to \ref{eAC}. Suppose further that $\norm{u_0}_\infty \leq 1$. Then, for all $t\in(0,\infty)$, there exists a $C\in(0,\infty)$ such that
    \begin{align*}
        \norm{\nabla_j u^\varepsilon(x,t)}\leq \frac{1}{\sqrt{t}}+\frac{\sqrt{t}}{\varepsilon^2}
    \end{align*}
    
\end{corollary}
\begin{proof}
    Using Duhamel's principle, we have
        \begin{align*}
            \partial_i u^{\varepsilon}(x,t)&=\int_{\bR^{n+1}} \partial^x_i \Phi(x-y,t)u_0(y) dy -\int_{0}^t \int_{\bR^{n+1}} \partial^x_i \Phi(x-y,t-s) \frac{f(u^\varepsilon(y,s))}{\varepsilon^2}dy ds\\
            &=\int_{\bR^{n+1}} \frac{y_i-x_i}{2t} \Phi(x-y,t)u_0(y) dy -\int_{0}^t \int_{\bR^{n+1}} \frac{y_i-x_i}{2(t-s)}\Phi(x-y,t-s) \frac{f(u^\varepsilon(y,s))}{\varepsilon^2}dy ds
        \end{align*}
        Since $\norm{u^\varepsilon}\leq 1$, after computing a standard Gaussian integral, we conclude the gradient is bounded by the above bound.

\end{proof}

As discussed in the introduction, in order to demonstrate that the limit as $\varepsilon\to0$ of the measures $\mu^\varepsilon$ is a Brakke flow, Ilmanen introduced a monotonicity formula. The convergence was demonstrated only when the discrepancy is non-positive.
 In \cite{Soner1}, Soner was able to extend this result to more general solutions by demonstrating that the discrepancy decays rapidly after the initial time, and is certainly non-positive in the $\varepsilon\to 0$ limit. 
Soner's results require the following assumption
    \begin{assumption}\label{Assumption1}
        For every $\delta>0$ there are positive constants $K_\delta$ and $\eta$ such that for every continuous function $\phi$,
        \begin{align*}
            \sup_{\varepsilon\in (0,1), t\in \left[ \delta, \frac1{\delta} \right]} \left\{ \int \norm{\phi(x)}\mu^\varepsilon(dx,t) \right\} \leq K_\delta \sup \{\norm{\phi(x)}e^{\eta\norm{x}}: x\in \bR^{n+1}\}.
        \end{align*}
    \end{assumption}
    In \cite{Soner2}, Soner verifies this assumption is satisfied for initial data derived from a $C^2$ hypersurface. 
    Combining the results of Ilmanen, Soner and Tonegawa, we get the following result concerning the limit of the energy measures as $\varepsilon\to 0$.
\begin{theorem}[Integral Brakke flow limits]\cite{Ilmanen,Soner1,Tonegawa}\label{bflimit}
    Let $M^n\subset\bR^{n+1}$ be a (generalised) hypersurface and $\Omega$ the region it bounds. Let $u_0\in L^\infty(\bR^{n+1})$ be a function that is positive on $\Omega$ and negative on $\Omega^c$, and denote by $u^\varepsilon$ the solution to the $\varepsilon$-Allen--Cahn flow starting from $u_0$ for each $\varepsilon\in(0,1)$.

    Suppose that the energy measures $\mu^\varepsilon(dx,t)$ defined by the solution $u^\varepsilon$ all satisfy Assumption \ref{Assumption1}.
 Then, taking the (subsequential) limit $\varepsilon\to 0$, we find a limiting family of measures $\{\mu_t\}_{t\in(0,\infty)}$, defining an integral Brakke flow.
\end{theorem}

\section{Barriers and Energy Estimates}
In order to apply Theorem \ref{bflimit}, we directly show the $\varepsilon$-AC energy is instantaneously bounded by $\frac{C}{\sqrt{t}}$, for initial data derived from $2\chi-1$, where $\chi$ is the indicator function for the compact region bounded by our generalised hypersurface, from which Assumption \ref{Assumption1} follows immediately.
    We note it should be possible to verify Assumption \ref{Assumption1} for non-compact hypersurfaces using these methods.

    For $\varepsilon\in(0,1)$, we demonstrate barriers to the solution $u_\varepsilon$ from our choice of $u_0$. Using these barriers, we can estimate the $\varepsilon$-Allen--Cahn energy.
\subsection{Barriers}

\begin{definition}
    Let $d(x,M)$ denote the signed distance to a (generalised) hypersurface $M$, we define $\hd$ as the heat flow of $d$:
        \begin{align*}
            \hd(x,t)  = \int_{\bR^{n+1}}\Phi(x-y,t)d(y,M)d\mu(y),
        \end{align*}
    where $\Phi$ is the heat kernel on $\bR^{n+1}$.
\end{definition}

It is clear that $\hd$ is a smoothing of the distance function. We record its properties.
\begin{proposition}[Properties of $\hd$] We have:
            \begin{enumerate}
                \item $(\pd{}{t} - \Delta) \hd(x,t) =0$,
                \item $|D \hd|(x,t)\leq 1$
                \item $\norm{d(x,M)-\hd(x,t)}\leq \sqrt{(n+1)t}$
            \end{enumerate}
\end{proposition}
\begin{proof}
    Item 1 is obvious from the definition of $\hd(x,t)$. To see Item 2, we first calculate the following:
        \begin{align*}
            \norm{\partial^x_i \hd(x,t)}^2=&\norm{\int_{\bR^{n+1}}\partial^x_i\Phi(x-y,t)d(y,M)d\mu(y)}^2\\
            =&\norm{\int_{\bR^{n+1}}\Phi(x-y,t)\partial^y_id(y,M)d\mu(y)}^2\\
            =&\norm{\int_{\bR^{n+1}}\Phi^{1/2}(x-y,t)\Phi^{1/2}(x-y,t)\partial^y_id(y,M)d\mu(y)}^2\\
            \leq&\norm{\int_{\bR^{n+1}}\Phi(x-y,t)d\mu(y)}\norm{\int_{\bR^{n+1}}\Phi(x-y,t)(\partial^y_id(y,M))^2d\mu(y)}\\
            =&\int_{\bR^{n+1}}\Phi(x-y,t)(\partial^y_id(y,M))^2d\mu(y)
        \end{align*}
        We note $d(x,M)$ is Lipschitz, with constant 1. Appealing to Sard's theorem, we conclude the derivative exists almost everywhere and has size 1. Item 2 follows by summing the above inequality for $i=1$ to $n+1$.

    Finally, Item 3:
        \begin{align*}
            \norm{d(x,M)-\hd(x,t)}=&\norm{d(x,M)-\int_{\bR^{n+1}}\Phi(x-y,t)d(y,M)d\mu(y)}\\
            =&\norm{\int_{\bR^{n+1}}\Phi(x-y,t)(d(x,M)-d(y,M))d\mu(y)}\\
            \leq& \int_{\bR^{n+1}}\Phi(x-y,t)\norm{(d(x,M)-d(y,M))}d\mu(y)\\
            \leq &\int_{\bR^{n+1}}\Phi(x-y,t)\norm{x-y}d\mu(y)\\
        \end{align*}
        using the change of variables $(y-x)=\sqrt{4t}z$
        \begin{align*}
            = \int_{\bR^{n+1}}\Phi(z,1)2\sqrt{t}\norm{z}d\mu(z)
            \leq 2\sqrt{t}\sqrt{\int_{\bR^{n+1}}\Phi(z,1)|z|^2 d\mu(z)}
            =\sqrt{(n+1)t}
        \end{align*}
\end{proof}

\begin{theorem} Let
    \begin{align*}
        g(x,t)=\tanh\left(\frac{\hd(x,t)}{\varepsilon}\right).
    \end{align*}
    Then, the function $g(x,t)$ is a sub-solution on $\hd(x,t)>0$ and a super-solution on $\hd(x,t)<0$. 
\end{theorem}
\begin{proof}
    We calculate $(\pd{}{t}-\Delta)g + \frac{2}{\varepsilon^2}f(g)$.
        \begin{align*}
            \partial_t g(x,t)&=\frac{\partial_t \hd}{\varepsilon} \operatorname{sech}^2\left(\frac{\hd(x,t)}{\varepsilon}\right)\\
            \partial_i g(x,t)&=\frac{\partial_i \hd}{\varepsilon} \operatorname{sech}^2\left(\frac{\hd(x,t)}{\varepsilon}\right)\\
            \partial_{ii} g(x,t)&=\frac{\partial_{ii} \hd}{\varepsilon} \operatorname{sech}^2\left(\frac{\hd(x,t)}{\varepsilon}\right)-2 \left(\frac{\partial_i \hd}{\varepsilon}\right)^2\operatorname{sech}^2\left(\frac{\hd(x,t)}{\varepsilon}\right)\tanh\left(\frac{\hd(x,t)}{\varepsilon}\right)
        \end{align*}
    from which we conclude 
        \begin{align*}
            (\pd{}{t}-\Delta)g + \frac{2}{\varepsilon^2}f(g)=\frac{2}{\varepsilon^2}\left( (\partial_i \hd)^2-1\right)\operatorname{sech}^2\left(\frac{\hd(x,t)}{\varepsilon}\right)\tanh\left(\frac{\hd(x,t)}{\varepsilon}\right)
        \end{align*}
    Since $(\partial_i\hd)^2\leq1$, we conclude the theorem, since $\tanh\left(y\right)$ has the same sign as its argument.
\end{proof}

\begin{theorem} Let
    \begin{align*}
        \tilde{g}(x,t)=\tanh\left(\frac{\hd(x,t)}{\sqrt{t}}\right).
    \end{align*}
    Then, the function $\tilde{g}(x,t)$ is a sub solution on $\hd(x,t)>4\sqrt{t}$ and a super solution on $\hd(x,t)<-4\sqrt{t}$. 
\end{theorem}
\begin{proof}
    We calculate $(\pd{}{t}-\Delta)\tilde{g} + \frac{2}{\varepsilon^2}f(\tilde{g})$.
        \begin{align*}
            \partial_t g(x,t)& =\left(\frac{\partial_t \hd(x,t)}{\sqrt{t}}-\frac{\hd(x,t)}{2t^{3/2}} \right)\operatorname{sech}^2\left(\frac{\hd(x,t)}{\sqrt{t}}\right)\\
            \partial_i g(x,t)&=\frac{\partial_i \hd(x,t)}{\sqrt{t}} \operatorname{sech}^2\left(\frac{\hd(x,t)}{\sqrt{t}}\right)\\
            \partial_{ii} g(x,t)&=\frac{\partial_{ii} \hd(x,t)}{\sqrt{t}} \operatorname{sech}^2\left(\frac{\hd(x,t)}{\sqrt{t}}\right)-2 \left(\frac{\partial_i \hd(x,t)}{\sqrt{t}}\right)^2\operatorname{sech}^2\left(\frac{\hd(x,t)}{\sqrt{t}}\right)\tanh\left(\frac{\hd(x,t)}{\sqrt{t}}\right)
        \end{align*}
    from which we conclude 
        \begin{align*}
            (\pd{}{t}-\Delta)g+\frac{2}{\varepsilon^2} = \left(-\frac{\hd(x,t)}{2t^{3/2}}+2\left(\frac{\norm{\nabla\hd} ^2(x,t)}{t}-\frac1{\varepsilon^2}\right)\tanh\left(\frac{\hd(x,t)}{\sqrt{t}}\right)\right)\operatorname{sech}^2\left(\frac{\hd(x,t)}{\sqrt{t}}\right)
        \end{align*}
    Since $(\partial_i\hd)^2\leq1$, we conclude the theorem, since $\tanh\left(y\right)$ has the same sign as its argument.
\end{proof}
\begin{remark}
    The same argument, and same distance restrictions, works for 
    \begin{align*}
        \tilde{g}(x,t)=\tanh\left(\frac{\hd(x,t)}{\sqrt{t}}-C\right)
    \end{align*}
    where $C\in \bR$ is constant.
\end{remark}
We will be integrating over the level-sets of the distance function, thus, it is more convient to have barriers in terms of the distance, rather than the heat flow of the distance. To this end:
\begin{lemma}[$u$ barrier Lemma]\label{ubarrier} Let $u$ solve \ref{eAC} with initial data $u_0$.
    On $d(x,M)\geq 5\sqrt{n+1}\sqrt{t}$
    \begin{align}
        u(x,t)\geq \tanh\left(\frac{d(x,M)}{\sqrt{t}}-5\sqrt{n+1}\right)
    \end{align}
    and on $d(x,M)\leq -5\sqrt{n+1}\sqrt{t}$
    \begin{align}
        u(x,t)\leq \tanh\left(\frac{d(x,M)}{\sqrt{t}}+5\sqrt{n+1}\right)
    \end{align}
    
\end{lemma}

\begin{proof}
    We first recall that 
        \begin{align*}
            \tilde{g}(x,t)=\tanh\left(\frac{\hd(x,M)}{\sqrt{t}}-4\sqrt{n+1}\right)
        \end{align*}
    is a sub-solution on $\hd(x,t)\geq 4\sqrt{n+1}\sqrt{t}$. Using an elementry comparison to the heat flow, we conclude $u>0$ on $d(x,M)\geq2\sqrt{n+1}\sqrt{t}$, thus, on the boundary $\hd(x,t)=4\sqrt{n+1}\sqrt{t}$, $\tilde{g}(x,t)\leq u(x,t)$. We deduce $\tilde{g}$ is a barrier to $u$ from below on this region.
    To conclude the theorem, we observe, 
        \begin{align}
            \tanh\left(\frac{d(x,M)}{\sqrt{t}}-5\sqrt{n+1}\right)\leq \tanh\left(\frac{\hd(x,t)}{\sqrt{t}}-4\sqrt{n+1}\right)
        \end{align}
        where we have used $\hd(x,t)\geq d(x,M)-\sqrt{n+1}\sqrt{t}$.
\end{proof}

\begin{definition}
    As demonstrated by Soner, \cite{Soner2}, $\norm{Du}^2$ is a subsolution to the following second order parabolic equation.
    \begin{align}
       \left(\partial_t-\Delta+ \frac{2}{\varepsilon^2}(3u^2-1)\right) \phi = 0 \label{deriveqn}
    \end{align}
\end{definition}

\begin{definition} We define the function
    \begin{align}
        \psi = \frac1{t}\exp\left(-\frac{\hd^2(x,t)}{2t}\right)
    \end{align}
\end{definition}
\begin{theorem}
    The function $\psi$ is a super-solution to \ref{deriveqn} on $|\hd|>5\sqrt{t}$.
\end{theorem}
\begin{remark}
    It should be possible to refine the distance, though it is not needed for our argument.
\end{remark}
\begin{proof}
    We calculate
        \begin{align*}
            \partial_t \psi &= \left(-\frac{1}{t^2}-\frac{\hd\partial_t\hd}{t^2}+\frac{\hd^2}{2t^3}\right)\exp\left(-\frac{\hd^2(x,t)}{2t}\right)\\
            \partial_i \psi &= \left(-\frac{\hd\partial_i\hd}{t^2}\right)\exp\left(-\frac{\hd^2(x,t)}{2t}\right)\\
            \partial_{ii} \psi &= \left(-\frac{\hd\partial_{ii}\hd}{t^2}  -\frac{(\partial_i\hd)^2}{t^2}+\frac{(\hd\partial_i\hd)^2}{2t^3}\right)\exp\left(-\frac{\hd^2(x,t)}{2t}\right)
        \end{align*}
    Hence,
        \begin{align*}
            (\partial_t-\Delta)\psi=\left(\frac{(\partial_i\hd)^2-1}{t^2}+ \hd^2\frac{1-(\partial_i\hd)^2}{2t^3}\right)\exp\left(-\frac{\hd^2(x,t)}{2t}\right)
        \end{align*}
        Since $1-(\partial_i \hd)^2\geq 1$, if $\hd^2>2t$, then $\psi$ is a super-caloric function. In order to be a solution to \ref{deriveqn}, we need to control the final term. 
        Following the proof of Lemma \ref{ubarrier} we may estimate the solution $u$ using the barrier $\tilde{g}$. We deduce $\norm{u}>\tanh(1)>\frac1{\sqrt{3}}$ on $|\hd|\geq 5\sqrt{t}$, thus $3u^2-1>0$, from which we conclude the theorem.
\end{proof}
\begin{lemma}[Gradient Barrier Lemma]\label{gradientbarrier} Let $u$ solve \ref{eAC} with initial data $u_0$. Then
    On $|d(x,M)|\geq 6\sqrt{n+1}\sqrt{t}$
    \begin{align}
        |Du|^2(x,t)\leq \frac{C}{t}\exp\left(-\frac{d^2(x,M)}{3t}\right)
    \end{align}
    for some constant $C$ depending on the dimension.
    
\end{lemma}
\begin{proof}
    Recall, $\psi$ is a super-solution on $\hd(x,t)\geq 5\sqrt{t}$, so we can find a choice of $\tilde{C}$ such that $\tilde{C}\psi$ is a barrier:
    The gradient satisfies
        \begin{align*}
            |Du|^2\leq \left(\frac{1}{\sqrt{t}}+\frac{\sqrt{t}}{\varepsilon^2}\right)^2
        \end{align*}
    on $(0,\varepsilon^2)$, we deduce
    \begin{align*}
        |Du|^2\leq \frac{2}{t}
    \end{align*}
    Setting $\tilde{C}=2\exp(\frac{25(n+1)}{2})$ we have, on $|\hd(x,t)|=5\sqrt{n+1}\sqrt{t}$
        \begin{align*}
            \frac{\tilde{C}}{t}\psi=\frac{2}{t}\exp\left(\frac{25(n+1)}{2}-\frac{25(n+1)}{2}\right)=\frac{2}{t}
        \end{align*}
    Since the gradient is initially 0 away from $M$, and we have control on the $\hd = 5\sqrt{n+1}\sqrt{t}$ boundary, we see $\tilde{C}\psi$ is a barrier on $\hd\geq 5\sqrt{n+1}\sqrt{t}$.
    
    To switch to the distance $d$, we recall $\norm{d(x,M)-\hd(x,t)}\leq \sqrt{n+1}\sqrt{t}$, and thus we must consider the region $d(x,M)\geq 6\sqrt{n+1}\sqrt{t}$ in order for $\psi$ to be a super solution.
    
    In order to change the argument from $\hd$ to $d$, we note
    \begin{align*}
        -\hd(x,t)^2&\leq -d(x,M)^2+2d(x,M)\sqrt{n+1}\sqrt{t}-(n+1)t\\
        &\leq -\frac{2}{3}d(x,M)^2-(n+1)t
    \end{align*}
    where we have used $|d(x,M)|\geq 6\sqrt{n+1}\sqrt{t}$. We conclude the lemma by setting $C=\tilde{C}\exp(-\frac{n+1}{2})$.
\end{proof}
\subsection{Energy Estimates}
    We need the following results for level-sets of the distance function.
        \begin{theorem}\cite[Theorem 6]{Caraballo}\label{n-1est}
            Let $M$ be a generalised hypersurface. i.e. $M\subset \bR^{n+1}$ compact, with $\cH^{n}(M)<\infty$, denote by $\Gamma_r=\{x\in \bR^{n+1}: d(x,M)=r\}$, the Euclidean distance to $M$. Suppose $M$ satisfies the following lower density bound: there exists $\delta,\theta>0$ such that
                \begin{align*}
                    \frac{\cH^{n}(M\cap B^{n+1}(p,\gamma))}{\gamma^{n-1}}\geq \theta
                \end{align*}
            for all $p \in M$ and $\gamma\in(0,\delta]$. Then, for $\cL^1$-almost every $r\in(0,\infty)$,
                \begin{align*}
                    \cH^{n}(\Gamma_r)\leq \left(\frac{\mu_{n+1}\beta(n+1)}{\theta}\right)\sup\left\{1, \left(\frac{r}{\delta}\right)^n\right\} \cH^{n}(M)
                \end{align*}
            Where $\beta(n+1)$ is the optimal constant from the Besicovitch Covering Theorem and 
            \begin{align*}
            \mu_{n+1}=\alpha(n+1)4^{n+1}(1+\sqrt{n})^2,
            \end{align*}
             where $\alpha(n+1)$ is the volume of the unit ball.
        \end{theorem}
    \begin{remark}
        We quote the above result as it has a better constant compared to Almgren--Taylor--Wang \cite{ATW}.
    \end{remark}

    We can now estimate the energy.

    \begin{theorem}\label{finiteenergy} Let $M$ be a generalised hypersurface with $\cH^{n}(M)<\infty$. Then, for all $\varepsilon\in(0,1)$ and $t\in(0,\varepsilon^2)$, we have
        \begin{align*}
            \int_{\bR^{n+1}}\varepsilon \frac{\norm{Du}^2}{2}+\frac{(u^2-1)^2}{2\varepsilon} d x \leq C \left(\frac{\varepsilon}{\sqrt{t}}+\frac{\sqrt{t}}{\varepsilon}\right) 
        \end{align*}
        Where $C<\infty$ depends on the dimension, $\cH^{n}(M)$, and the lower density bound on $M$.
    \end{theorem}

    \begin{remark}
        Observe, for $\varepsilon\in(0,1)$, and  $t\in(0,\varepsilon^2)$, the upper bound can be bounded by $\frac{2C}{\sqrt{t}}$, which is independent of $\varepsilon$.  At time $t=\varepsilon^2$, the upper bound is equal to $2C$, and since the flow is the gradient flow for the energy, this bound is true for all $t>\varepsilon^2$.
    \end{remark}
    \begin{proof} Our barrier functions are written as functions of this distance. Since the distance function $d(x,M)$ has $\norm{\nabla d(x,M)}=1$ $dx$-a.e. and the level sets have finite $\cH^{n}$ measure (Theorem \ref{n-1est}), we can evaluate the integrals via the co-area formula.
        \begin{remark}
            To streamline the exposition, for the application of Theorem \ref{n-1est} we assume $\delta=1$ and $\varepsilon<\frac1{6\sqrt{n+1}}$. To deal with the other cases, one simply has to split the integral differently.
        \end{remark}
        First, we evaluate the kinetic energy.  We split the integral into two regions:
        \begin{align*}
            \int_{\bR^{n+1}}\varepsilon \frac{\norm{Du}^2}{2}dx= \int_{\norm{d(x,M)}\leq 6\sqrt{n+1}\sqrt{t}}\varepsilon \frac{\norm{Du}^2}{2}dx + \int_{\norm{d(x,M)}\geq 6\sqrt{n+1}\sqrt{t}}\varepsilon \frac{\norm{Du}^2}{2}dx
        \end{align*}
        In the region near the initial hypersurface, we use the a priori derivative estimate:
        \begin{align*}
            \int_{\norm{d(x,M)}\leq 6\sqrt{n+1}\sqrt{t}}\varepsilon \frac{\norm{Du}^2}{2}dx&\leq \int_{-6\sqrt{n+1}\sqrt{t}}^{6\sqrt{n+1}\sqrt{t}}\int_{\rho^{-1}(\{s\})}\frac{\varepsilon}{2} \left(\frac{1}{t}+\frac{t}{\varepsilon^2}+\frac1{\varepsilon^2}\right)\cH^{n}(\Gamma_s)ds\\
            &=C\varepsilon\sqrt{t}\left(\frac{1}{t}+\frac{t}{\varepsilon^2}+\frac1{\varepsilon^2}\right)
        \end{align*}
        In the other region, we use the barrier from Lemma \ref{gradientbarrier}
        \begin{align*}
            \int_{\norm{d(x,M)}\geq 6\sqrt{n+1}\sqrt{t}}\varepsilon \frac{\norm{Du}^2}{2}dx&\leq \int_{\norm{d(x,M)}\geq 6\sqrt{n+1}\sqrt{t}} \frac{C\varepsilon}{t}\exp\left(-\frac{d^2(x,M)}{3t}\right)dx\\
            &\leq\frac{C\varepsilon}{t}\int_{6\sqrt{n+1}\sqrt{t}}^\infty\int_{\rho^{-1}({s})} \exp\left(\frac{-s^2}{3t}\right)\cH^{n}\left(\Gamma_s\right)ds
        \end{align*}
        We split into the two cases of the supremum in Lemma \ref{n-1est}.
        
        First from $6\sqrt{n+1}\sqrt{t}$ to 1:
        \begin{align*}
            \frac{C\varepsilon}{t}\int_{6\sqrt{n+1}\sqrt{t}}^1\int_{\rho^{-1}({s})} \exp\left(\frac{-s^2}{3t}\right)\cH^{n}\left(\Gamma_s\right)ds=\frac{C\varepsilon}{\sqrt{t}}\int_{2\sqrt{n+1}\sqrt{3}}^\infty\exp\left(-v^2\right)dv
        \end{align*}
        and then from $1$ to $\infty$:
        \begin{align*}
            \frac{C\varepsilon}{t}\int_{1}^\infty\int_{\rho^{-1}({s})} \exp\left(\frac{-s^2}{3t}\right)\cH^{n}\left(\Gamma_s\right)ds&=\frac{C\varepsilon}{t}\int_{1}^\infty\int_{\rho^{-1}({s})} \exp\left(\frac{-s^2}{3t}\right)s^{n}\cH^{n}\left(M\right)ds\\
            &=C\varepsilon \left(\sqrt{t}\right)^{n-2}\int_{\frac1{\sqrt{3t}}}^\infty v^{n}\exp\left(-v^2\right)   dv\leq \frac{C\varepsilon}{\sqrt{t}}.
        \end{align*}

        Finally, we estimate the potential energy, this time spliting the integral at  $\norm{d(x,M)}=5\sqrt{n+1}\sqrt{t}$.
        \begin{align*}
            \int_{\bR^{n+1}}\frac{(u^2-1)^2}{2\varepsilon} dx= \int_{\norm{d(x,M)}\leq 5\sqrt{n+1}\sqrt{t}}\frac{(u^2-1)^2}{2\varepsilon} dx+ \int_{\norm{d(x,M)}\geq 5\sqrt{n+1}\sqrt{t}}\frac{(u^2-1)^2}{2\varepsilon} dx
        \end{align*}
         In the region near the hypersurface, we can simply bound $(u^2-1)^2\leq1$, yielding:
        \begin{align*}
            \int_{\norm{d(x,M)}\leq 5\sqrt{n+1}\sqrt{t}}\frac{(u^2-1)^2}{2\varepsilon} dx \leq C\frac{\sqrt{t}}{\varepsilon}
        \end{align*}
        Using Lemma \ref{ubarrier}, we have $u(x,t)>\norm{\tanh\left(\frac{d(x,M)}{\sqrt{t}}-5\sqrt{n+1}\right)}$ on $d(x,M)\geq 5\sqrt{n+1}\sqrt{t}$, and $\norm{u}<1$. Thus:
        \begin{align*}
            \int_{d(x,M)\geq 5\sqrt{n+1}\sqrt{t}}\frac{(u^2-1)^2}{2\varepsilon} &\leq \int_{d(x,M)\geq 5\sqrt{n+1}\sqrt{t}}\frac{(\tanh^2\left(\frac{d(x,M)}{\sqrt{t}}-5\sqrt{n+1}\right)-1)^2}{2\varepsilon} dx\\
            &=\int_{d(x,M)\geq 5\sqrt{n+1}\sqrt{t}}\frac{\operatorname{sech}^4\left(\frac{d(x,M)}{\sqrt{t}}-5\sqrt{n+1}\right)}{2\varepsilon} dx\leq C\frac{\sqrt{t}}{\varepsilon}.
        \end{align*}
        Using Lemma \ref{ubarrier} again, we can, in exactly the same manner, estimate the integral on $d(x,M)\leq -5\sqrt{n+1}\sqrt{t}$.

    \end{proof}

\subsection{Sets with Hausdorff dimension $n+\kappa$}
    In this section, we assume $\kappa\in(0,1)$ .
    
    When $\kappa=1$, we have $\int_a^b 1/s^\kappa ds = \left[\log(s)\right]_a^b$, and thus one cannot integrate from 0 to 1. 
    
    We first prove a fractional dimension version of Lemma \ref{n-1est}
    \begin{lemma}\label{nkappa}
        Suppose $\Omega$ is bounded, with $\cH^{n+\kappa}(\Omega)\in(0,\infty)$ for some $\kappa\in(0,1)$. Suppose further there exists a lower $n+\kappa$ density bound, i.e. there is some $\theta>0$ and $\delta>0$ such that 
        \begin{align*}
            \frac{\cH^{n+\kappa}(\Omega\cap B(x,\gamma))}{\gamma^{n+\kappa}}\geq \theta.
        \end{align*}
        for all $x\in M$ $\gamma\in(0,\delta]$.
        
        Then, for $r\in(0,\infty)$, we have
        \begin{align*}
            \cH^{n}(\Gamma_r)\leq  \left(\frac{\mu_{n+1}\beta(n+1)}{\theta}\right) \sup\left\{r^{-\kappa}, \left(\frac{r^{n}}{\delta^{n+\kappa}}\right)\right\}\cH^{n+\kappa}(M).
        \end{align*}
    \end{lemma}
    \begin{remark}
       Let $m=n, s=n+\kappa$. Then, an $(m/s)$-H\"older immersion of $M^n\to \bR^{n+1}$ should satisfy this condition, for any $n$-manifold $M$. These are examples of the fractional $(m/s)$-rectifiable sets of Mart\'in and Mattila,\cite{MartinMattila}. See also the work of Badger---Vellis \cite{BadgerVellis}. 
    \end{remark}
    
    \begin{proof}
        We follow the proof of  \cite[Theorem 6]{Caraballo}. 
        We consider the collection $\cF^*$ of balls of radius $r>0$ centered at each $p\in \Omega$. By Besicovitch covering theorem, and the compactness of $\Omega$, we can find a finite subcover (with centres $p_i$) such that 
            \begin{align*}
                \sum^N_i \cH^{n+\kappa}\left( \Omega\cap B^{n+1}(p_i, r)\right)\leq \beta(n+1)\cH^{n+k}(\Omega)
            \end{align*}
        where $\beta(n+1)$ is the Besicovitch constant.
        We now observe, by the triangle inequality
        \begin{align*}
        \rho^{-1}(\{r\})\subset \bigcup^N_{i} B(p_i, 2r)
        \end{align*}
        and hence, by \cite[Proposition 13]{Caraballo}, 
            \begin{align*}
                \cH^{n}(\rho^{-1}(\{r\})) \leq \sum^N_i \cH^{n}\left( \rho^{-1}(\{s\})\cap B^{n+1}(p_i, 2r)\right) \leq \sum^N_i \mu_{n+1}r^{n}
            \end{align*}
        Now, if $r<\delta$, we deduce from the lower density bound:
            \begin{align*}
                r^{n+\kappa}\leq\frac{1}{\theta}\cH^{n+\kappa}\left( \Omega\cap B^{n+1}(p_i, r)\right).
            \end{align*}
        from which we conclude 
            \begin{align*}
                \cH^{n}(\rho^{-1}(\{r\})) \leq \frac{\mu_{n+1}\beta(n+1)}{\theta}\cH^{n+\kappa}(\Omega)  r^{-\kappa}.
            \end{align*}
        If $r>\delta$, we do the following algebraic manipulation:
            \begin{align*}
                r^{n+\kappa}&=\left(\frac{r}{\delta}\right)^{n+\kappa}\delta^{n+\kappa}\\
                &\leq\left(\frac{r}{\delta}\right)^{n+\kappa}\frac{\cH^{n+\kappa}\left( \Omega\cap B^{n+1}(p_i, \delta)\right)}{\theta}\\
                &\leq\left(\frac{r}{\delta}\right)^{n+\kappa}\frac{\cH^{n+\kappa}\left( \Omega\cap B^{n+1}(p_i, r)\right)}{\theta}
            \end{align*}
            
    \end{proof}

    We can then repeat the integral estimates using this estimate on the level-sets, yielding

    \begin{theorem}\label{finiteenergy2}
        Let $M$ be a generalised hypersurface satisfying the assumptions of Lemma \ref{nkappa}, then
        \begin{align*}
            \int_{\bR^{n+1}}\varepsilon \frac{\norm{Du}^2}{2}+\frac{(u^2-1)^2}{2\varepsilon} d x \leq C \varepsilon t^{-\frac{1+\kappa}{2}}
        \end{align*}
    \end{theorem}
    Then, applying the results of Ilmanen, Soner, and Tonegawa
    \begin{corollary}
        After multiplying by a constant, the limit as $\varepsilon\to0$ of the energy measures is an integral Brakke flow starting from the generalised hypersurface $M$.
    \end{corollary}
    \begin{remark}
        Here `starting from' means that the Hausdorff distance of the support goes to $0$ as $t\to0^+$.
    \end{remark}
    \begin{theorem}\label{goodstartingfrom}
        If $M$ is a smooth, embedded compact hypersurface, then each solution $u_\varepsilon$ is multiplicity one for some short time, and the $\varepsilon\to 0 $ limit of the energy measures converges to a unit-regular, integral Brakke starting from $M$ (in the classical sense) with multiplicity one, as long such a multiplicity one flow exists.
    \end{theorem}
    \begin{proof}
        This follows immediately from the convolution definition of the $\varepsilon$-AC flow. Indeed, blowing up any solution to the $\varepsilon$-AC flow at any point of the hypersurface yields a solution to the heat equation. Since the hypersurface is smooth, the initial data blows up to $+1$ on the upper half space, and $-1$ on the lower half space (up to rigid Euclidean motion). Thus, the limiting solution is the 1-d sigmoidal solution to the heat equation product a $n-1$-plane.

        This is true at every point, and since the hypersurface is compact and the convergence is smooth, this can be seen at a uniform scale across the hypersurface. Thus, for short time, there are only two contiguous phases, and there cannot be higher multiplicity. The convergence is thus with multiplicity one, as long as the Brakke flow remains multiplicity one - this is atleast until the first singular time.
    \end{proof}
    \begin{remark}\label{goodstartingfromremark}
        In this case, we can really extend the limiting Brakke flow to time 0 setting $\mu_0(\cdot) = \cH^n\lfloor M_0(\cdot)$, with no issue.
    \end{remark}

\section{The Family of Allen-Cahn solutions}
    In this section, $M^n\subset \bR^{n+1}$ is a closed hypersurface (either smooth or generalised). We consider a tubular neighbourhood of $M$, $T_\eta(M):= \{x\in \bR^{n+1}, |d(x,M)|\leq \eta\}$ and foliate $T_\eta(M)$ by (generalised) hypersurfaces $\{M_s\}_{s\in[-\eta,\eta]}$,  $M_s=\{x: d(x,M)=s\}$. 
    Note, if $M$ is smooth then it is well-known that $\eta>0$ can be chosen sufficiently small that the leaves $M_s=\{x: d(x,M)=s\}$ are smooth for $s\in[-\eta,\eta]$.
    \begin{remark}
        We take the convention that the signed distance function is positive on the interior. This gives 
            \begin{align*}
                \mathrm{sgn}(d(M,\cdot))=2\chi_{M}-1.
            \end{align*}
    \end{remark}
    We choose $\eta$ such that $\lsf(M_{\pm\eta})$ does not develop an interior (non-fattening). 
        \begin{theorem}\label{contmax}
            For each $\varepsilon\in(0,1)$, there exists a continuous 1-parameter family $\{u^\varepsilon_s\}_{s\in[-\eta,\eta]}$, where each $u^\varepsilon_s$ is the unique solution to the initial value problem
                \begin{align*}
                    \pd{}{t}u^\varepsilon_s&=\Delta u^\varepsilon_s -\frac1{\varepsilon} f(u^\varepsilon_s)\\
                    u^\varepsilon_s(\cdot, 0)&= 2\chi_{M_s}(\cdot)-1
                \end{align*}
            Moreover, if $s_1\leq s_2$, then 
                \begin{align*}
                    u^\varepsilon_{s_2}(x,t)\leq  u^\varepsilon_{s_1}(x,t)
                \end{align*}
        \end{theorem}
        \begin{proof}
            The existence of unique solutions from data in $L^\infty(\bR^{n+1})$ was established in \cite{MS1}, see Theorem \ref{exist}. Moreover, they establish both continuity (in the topology of distributions) of the solutions in the initial data, as well as a maximum principle (\cite[Equation 2.15]{MS1}) from which the second part of the theorem follows immediately.
        \end{proof}
        \begin{proposition}\label{nodalset}
            Let $X=(x,t)\in \lsf(M_0),t>0$  be a point in the level-set flow (LSF) from $M_0$ and let $\{\varepsilon_j\}_{j\in\bN}$ be a sequence with $\varepsilon_j\to0$. Suppose $\{s_j\}_{j\in\bN}$ is a sequence such that
                \begin{align*}
                    u_j(X)=u_{s_j}^{\varepsilon_j}(X)=0.
                \end{align*}
            Let 
                \begin{align*}
                    \cM^j=\left\{ \left[\frac{\varepsilon^j \norm{D u^j}^2(x,t)}{2}+\frac{F(u^j(x,t))}{\varepsilon}\right]dx\right\}_{t\in(0,\infty)}
                \end{align*} 
            be the associated sequence of 1-parameter families of Radon measures.
            Then, any subsequential limiting family of Radon measures $\cM^\infty$ is (up to a constant) an integral Brakke flow, with
                \begin{align*}
                    X\in \supp{\cM^\infty}.
                \end{align*}
        \end{proposition}
        \begin{proof} By the compactness theorem for Radon measures, we can always extract subsequential limit $\cM^\infty$, $\cM^{j_i}\rightharpoonup \cM^\infty$. Appealing to Theorem \ref{finiteenergy},( or \ref{finiteenergy2}), we see can apply Theorem \ref{bflimit}, \cite{Ilmanen,Soner1,Tonegawa}, and conclude that $\cM^\infty$ must be an integral Brakke flow.

            Suppose now that the claim regarding $X$ does not hold. Then by item (ii) of Ilmanen's clearing out lemma \cite[Theorem 6.1]{Ilmanen}, we see the $u_j$ must converge uniformly to $+1$ or $-1$ in a neighbourhood of $X$, however, this contradicts our assumption that $u_j(X)=0$. Thus, $X$ must be in the support of $\cM^\infty$.
            
        \end{proof}
        \begin{corollary}
            Let $\cM^\infty$ be the limiting Brakke flow from Proposition \ref{nodalset}. Then,
                \begin{align*}
                    \supp{\cM^\infty}\subset\lsf(M_0).
                \end{align*}
            In particular, we have 
                \begin{align*}
                    \supp{\lim_{t\downarrow 0} d\mu^\infty(t)}\subset M_0
                \end{align*}
        \end{corollary}
        \begin{proof}
            The energy measure is ill-defined at time $0$ along our sequence, so we cannot argue that the initial energy measure converges $\cH^{n}\lfloor M_0$ (up to a constant) which would immediately yield that the (support of the) limiting Brakke flow is a weak set flow from $M_0$. 

            Instead, we use a barrier argument. Fix $s\in (0,+\eta$), and we may presume $s$ is chosen such that $\mathrm{LSF}(M_{\pm s})$ does not fatten. Then, for $\varepsilon<s/2$ we construct `standard' initial data $\tilde{u}_{\pm s}^\varepsilon$ for the hypersurfaces $M_{\pm s}$ as per Ilmanen \cite{Ilmanen} or Modica \cite{Modica}, with the additional constraint that for $r\in (0,s/4)$
                \begin{align*}
                    \tilde{u}_{+s}^\varepsilon(\cdot, 0) \leq u_r(\cdot, 0) \leq \tilde{u}_{- s}^\varepsilon(\cdot, 0).
                \end{align*}
            where $u_r= 2\chi_{M_r}-1$. By Evans--Soner--Souganidis \cite{ESS}, the support of the energy measures associated to $\tilde{u}_{\pm s}^\varepsilon(\cdot, t)$ converges, as $\varepsilon\to 0$ to the level-set flow of $M_{\pm s}$, since $s$ was chosen to be a non-fattening level.

            Since $s_j,\varepsilon_j\to 0 $ as $j\to \infty$, there exists $J\in \bN$ such that for $j\geq J$, $\varepsilon_j<s/2$, $s_j<s/4$. Thus, we have
                \begin{align*}
                    \tilde{u}^{\varepsilon_j}_{+s}(\cdot, 0)\leq u_{s_j}(\cdot, 0)\leq \tilde{u}^{\varepsilon_j}_{-s}(\cdot,0),
                \end{align*}
            and by the maximum principle, for $t>0$ we also have
                \begin{align*}
                    \tilde{u}^{\varepsilon_j}_{+s}(\cdot, t)\leq u^{\varepsilon_j}_{s_j}(\cdot, t)\leq \tilde{u}^{\varepsilon_j}_{-s}(\cdot,t).
                \end{align*}
            Taking $j\to\infty$ (and an application of Ilmanen's Clearing-Out Lemma), we conclude that the support of $\cM^\infty$ satisfies 
                \begin{align*}
                    \supp{\cM^\infty} \subset \lsf(\mathrm{int}(M_{-s}))\\
                    \supp{\cM^\infty} \subset \lsf (\mathrm{ext}(M_{+s})).
                \end{align*}
            Since $s$ was arbitrary, we conclude $\supp{\cM^\infty}\subset\lsf(M_0)$. The second claim follows immediately.
            
        \end{proof}
    \begin{remark}
        In general, we can now extend the Brakke flow $\cM^\infty$ to time 0 by setting $\mu^\infty_0=\cH^n \lfloor M_0$ at time 0. Note we do not demonstrate equality of the support at $t=0$, nor do we say anything about the mass, meaning the flow may exhibit instantaneous mass drop or gain, see Section \ref{thefigure8} for an example.
    \end{remark}
    \begin{theorem}\label{supported}
        Let $M_0$ be a compact (generalised) hypersurface. Suppose that $\lsf(M_0)$ fattens in finite time. Then, for each $X_0\in\lsf(M_0)$, there exists a unit-regular, integral Brakke flow $\{\mu_t\}_{t\in[0,\infty)}$ starting from $M_0$, supported at $X_0$.
    \end{theorem}

    \begin{remark}
        Recall, we mean the weak notion of `starting from' in terms of Hausdorff distances. Recall further, this can often be improved, in particular the case of smooth initial data, see Theorem \ref{goodstartingfrom}.
    \end{remark}
    \begin{proof}
        Let $\{M_s\}_{s\in[-\eta,\eta]}$ be the above foliation and let $X_0=(x_0,t_0)\in \lsf(M_0)$. We first consider end-points of the foliation. By the choice of foliation, $\lsf(M_{\pm\eta})$ do not fatten. Moreover, by the weak-set-flow characterisation of LSF by Ilmanen \cite{IlmanenER}, we have the set inclusions
            \begin{align*}
                \lsf(M_0)(t_0)&\subset \lsf(\mathrm{int}(M_{-\eta})(t_0))\\
                \lsf(M_0)(t_0)&\subset \lsf(\mathrm{ext}(M_{+\eta})(t_0)).
            \end{align*} 
        From these inclusions, we conclude the point $X_0$ will lie exterior to the region bounded by $M_{+\eta}(t_0)$ (the $t_0$ time slice of $\lsf(M_{+\eta})$) and interior to the region bounded by $M_{-\eta}(t_0)$. Thus, by Ilmanen's (or in fact Evans--Soner--Souganidis \cite{ESS}) convergence theorem for solutions to the $\varepsilon$-Allen--Cahn equation, for each $\kappa\in(0,\frac1{2})$, we can find a $\varepsilon_0>0$ such that if $\varepsilon\in(0,\varepsilon_0)$, we have
            \begin{align*}
                u^\varepsilon_{-\eta}(X_0)&\geq 1-\kappa\\
                u^\varepsilon_{+\eta}(X_0)&\leq \kappa-1.
            \end{align*}
        \begin{remark}
            Theorem \ref{contmax} also provides a maximum principle. We deduce that $s_0$ is either unique or there exists an interval $[s_1,s_2]\subset (-\eta,+\eta)$ on which $u^\varepsilon_s(X)=0$.
        \end{remark}
        We now apply Proposition \ref{nodalset}. We take a sequence $\varepsilon_j\to0$, and set $s_j=s_0(\varepsilon_j)$ as above.  Thus, by taking a subsequential limit, we find a limiting, unit-regular, integral Brakke flow $\cM_\infty$ starting from  $M_0$ that is supported at $X$.
    \end{proof}
\section{Examples}\label{examples}
We now collect some examples that demonstrate the applications of our results. These examples also highlight some interesting questions regarding the consequences of this work.
\subsection{The $g$-Wagon-Wheel of Ilmanen--White}

\begin{figure}[!ht]
    \centering
    \includegraphics[scale=0.2]{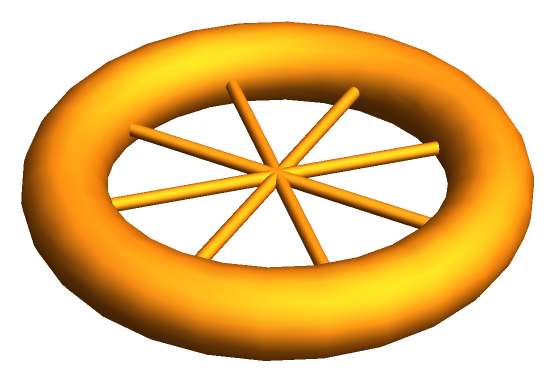}
    \textit{\caption{The 8-wheel of Ilmanen--White, Image reproduced from \cite{IW}\label{wheel}}}
\end{figure}

Consider the $g$-Wagon-Wheel, or $g$-wheel, of Ilamen--White, \cite[Definition 2.1]{IW}. The $8$-wheel is pictured in Figure \ref{wheel}. The definition of a $g$-wheel ensures the surface has dihedral symmetry in the $xy$-plane ($D_g$ invariant), and has an additional reflection symmetry across the $xy$-plane. Thus, it has symmetry group $D_g\times \mathbb{Z}_2$.

The work of Ilmanen--White shows that for a large value of $g$, the mean curvature flow of $M$ encounters a singularity at the origin at the first singular time. Moreover, this singularity is a conical, and for sufficiently large $g$ the model cone has sufficiently large cone angle that its forward evolution is non-unique. The work of Chodosh--Daniels-Holgate--Schulze, \cite{CDHS} hence shows the forward evolution is non-unique.

Ilmanen--White observed that the origin is contained in the LSF from $M$ from the first singular time, $T_1$ until the extinction of the outer flow, $T_\mathrm{ext}$. Thus, for each $t_0\in(T_1,T_\mathrm{ext})$ using Theorem \ref{supported}, we find a unit-regular integral Brakke flow, $\cM_{X_0}$, starting from the $g$-wheel that is supported at $X_0=(0,t_0)$. The initial data is smooth, so Theorem \ref{goodstartingfrom} means $\cM_{X_0}$ starts from $M$ in the classical sense.

We now show that each time-slice of $\cM_{X_0}$ also has $D_g\times \mathbb{Z}_2$-symmetry. Observe, we may foliate a neighbourhood of $M$ by surfaces that preserve the symmetry. Moreover, the symmetry is preserved by the $\varepsilon$-Allen--Cahn flow from each of these leaves, for every choice of $\varepsilon$. Thus, when taking limits in $s$ and $\varepsilon$, any limiting Brakke flow will inherit the symmetry. 

Recall, the support of $\cM_{X_0}$ at time $t_0$ includes the origin (by our theorem). The reflection symmetry in the $xy$-plane ensures this is a singular point: any `sheet' of the $t_0$ time-slice coming in to touch the origin not strictly contained in the $xy$-plane is reflected, so multiple `sheets' must touch at the origin. If the `sheet' is contained in the $xy$-plane, it must have higher multiplicity: the flow is a limit of solutions to the Allen--Cahn flow, and the reflection symmetry ensures that either side of the $xy$-plane must have the same phase. This is only possible with even multiplicity. This is perhaps easier to see in the example of the figure 8, see Section \ref{thefigure8}.

Whilst the resolution of the multiplicity-one conjecture by Bamler--Kleiner, \cite{BK}, applies to the outermost flows, it is presently unclear if the work of Bamler--Kleiner applies to the flow $\cM$. If it does, it rules out the higher multiplicity flow. It is a very interesting question to show whether or not the Bamler--Kleiner theory applies to the theory developed here, when starting from smooth surfaces.

\subsection{The Figure 8} \label{thefigure8}
\begin{figure}[!ht]
 
\centering
\begin{minipage}{0.3\textwidth}
    \centering

    \begin{tikzpicture}[scale =1.5]
    \def\a{3pt} \def\r{35pt} \def\rad{6pt}
        \draw[rounded corners=\rad,fill=white, draw=black, line width=1pt] (0,0) ++(45:1.1) arc (45:-45:1.1) -- (135:1.1) arc (135:225:1.1) -- (0,0) --cycle;
\end{tikzpicture}

\textit{\caption{The Figure 8 \label{fig8}}}
\end{minipage}
\hfill
 \begin{minipage}{0.3\textwidth}
\centering
\begin{tikzpicture}[scale = 1.5]

    \def\a{1pt} \def\r{35pt} \def\rad{8pt} \def\tip{0.4}
    \draw[rounded corners=\rad,fill=black!25, draw=black, line width=1pt] (0,0) ++ (45:1.1) arc (45:-45:1.1) -- (0,0) --+ (-135:1.1) arc (225:135:1.1) -- (0,0) --cycle  ;
        \draw[rounded corners=\rad,fill=white, draw=black, line width=1pt] (\tip,0) ++(45:0.6) arc (45:-45:0.6) -- (0,0)--cycle;
        \draw[rounded corners=\rad,fill=white, draw=black, line width=1pt] (-\tip,0) ++(135:0.6) arc (135:225:0.6) -- (0,0) --cycle;
        
\end{tikzpicture}

\textit{\caption{LSF a short time later\label{LSF8}}}
\end{minipage}
\hfill
 \begin{minipage}{0.3\textwidth}
\centering
\begin{tikzpicture}[scale = 1.3]
    
    \def\a{1pt} \def\r{35pt} \def\rad{6pt} \def\b{0.3}
        \draw[rounded corners=\rad,fill=white, draw=black, line width=1pt] (\b,0) --++(60:1) arc (60:-60:1) -- (\b,0);
        \draw[line width=1pt] (\b,0) --(-\b,0);
        \draw[rounded corners=\rad,fill=white, draw=black, line width=1pt] (-\b,0)-- ++(120:1) arc (120:240:1) -- (-\b,0) ;
        
\end{tikzpicture}

\textit{\caption{ The Horizontal\\ Two-Triple junction\label{twotriple}}}
\end{minipage}
\end{figure}

Consider now the Figure 8 curve (Figure \ref{fig8}), and its LSF, (Figure \ref{LSF8}). The figure 8 has $D_2$ symmetry.  Following roughly the same argument as for the $g$-wheel yields, for each time $t_0\in (0,T)$ where $T$ is the extinction time of the outer flow, a solution supported at $(0,t_0)$ . The $D_2$ symmetry ensures either we find the cross at time $t_0$, or we find a horizontal or vertical two-triple junction resolutions, as they also have $D_2$ symmetry and are supported at the origin, see Figure \ref{twotriple}.  However, the Allen-Cahn flow cannot produce Figure \ref{twotriple} (or its vertical equivalent) unless the middle bar has multiplicity $2n$, for $n\in \mathbb{Z}_{>0}$, as the phase does not change sign across it. 
The regularity theory of Kim--Tonegawa \cite{KimTonegawa} shows if we find the cross at time $t_0$, it existed for all previous times: if the cross regularises into two triple junctions or smooth curves, it cannot reform. In this case, it is possible to find a flow in which the cross persists from the initial time until the extinction time. 
Thus, the limit we find must preserve the cross until extinction, or resemble Figure \ref{twotriple} (or its vertical equivalent), with a even multiplicity middle bar. One might hope a monotonicity argument would rule out the two-triple junction resolutions with multiplicity strictly greater than 2 (they have density $\geq3$ at the junctions, larger than the initial maximum density of 2). Unfortunately, we cannot use monotonicity, as the theory developed here only attains the initial data in a Hausdorff sense. However, we note that the persistent cross and multiplicity 2 middle bar cases can be said to start from the figure 8 in the classical sense, as the densities agree at every point as $t\to 0^+$. 
 Instead of using monotonicity, we note that for each leaf of the initial folation, there are never more than two of each phase near the origin, ruling out multiplicity greater than 2. 
 It would be interesting to know which (or if all) of these flows exist. 
\vspace{-19pt}
\bibliographystyle{alpha}
\bibliography{biblio}

\begin{thebibliography}{{Whi}09}

\bibitem[AC79]{AllenCahn}
Samuel~M. Allen and John~W. Cahn.
\newblock A microscopic theory for antiphase boundary motion and its application to antiphase domain coarsening.
\newblock {\em Acta Metallurgica}, 27(6):1085--1095, 1979.

\bibitem[ATW93]{ATW}
Fred Almgren, Jean~E. Taylor, and Lihe Wang.
\newblock Curvature-driven flows: a variational approach.
\newblock {\em SIAM J. Control Optim.}, 31(2):387--438, 1993.

\bibitem[BCW24]{CBW}
Jacob Bernstein, Letian Chen, and Lu~Wang.
\newblock Existence of monotone {Morse} flow lines of the expander functional.
\newblock {\em \url{https://arxiv.org/abs/2404.08541},}, 2024.

\bibitem[BK23]{BK}
Richard~H Bamler and Bruce Kleiner.
\newblock On the {Multiplicity} {One} {Conjecture} for {Mean} {Curvature} {Flows} of surfaces.
\newblock {\em \url{https://arxiv.org/abs/2312.02106},}, 2023.

\bibitem[{Bra}78]{Brakke}
Kenneth~A. {Brakke}.
\newblock {The motion of a surface by its mean curvature}.
\newblock {Princeton, New Jersey: Princeton University Press. Tokyo: University of Tokyo Press}, 1978.

\bibitem[BV19]{BadgerVellis}
Matthew Badger and Vyron Vellis.
\newblock Geometry of measures in real dimensions via {H{\"o}lder} parameterizations.
\newblock {\em J. Geom. Anal.}, 29(2):1153--1192, 2019.

\bibitem[Car05]{Caraballo}
David~G. Caraballo.
\newblock Areas of level sets of distance functions induced by asymmetric norms.
\newblock {\em Pac. J. Math.}, 218(1):37--52, 2005.

\bibitem[CDHS24]{CDHS}
Otis Chodosh, J.~M. Daniels-Holgate, and Felix Schulze.
\newblock Mean curvature flow from conical singularities.
\newblock {\em Invent. Math.}, 238(3):1041--1066, 2024.

\bibitem[Che92]{Chen}
Xinfu Chen.
\newblock Generation and propagation of interfaces in reaction-diffusion systems.
\newblock {\em Trans. Am. Math. Soc.}, 334(2):877--913, 1992.

\bibitem[dMS95]{MS1}
Piero de~Mottoni and Michelle Schatzman.
\newblock Geometrical evolution of developed interfaces.
\newblock {\em Trans. Am. Math. Soc.}, 347(5):1533--1589, 1995.

\bibitem[ES91]{ES1}
L.~C. {Evans} and J.~{Spruck}.
\newblock {Motion of level sets by mean curvature. I}.
\newblock {\em {J. Differ. Geom.}}, 33(3):635--681, 1991.

\bibitem[ESS92]{ESS}
L.~C. Evans, H.~M. Soner, and P.~E. Souganidis.
\newblock Phase transitions and generalized motion by mean curvature.
\newblock {\em Commun. Pure Appl. Math.}, 45(9):1097--1123, 1992.

\bibitem[Her17]{Her_reif}
Or~Hershkovits.
\newblock Mean curvature flow of {R}eifenberg sets.
\newblock {\em Geom. Topol.}, 21(1):441--484, 2017.

\bibitem[HW20]{hershwhite}
Or~{Hershkovits} and Brian {White}.
\newblock {Nonfattening of mean curvature flow at singularities of mean convex type}.
\newblock {\em {Commun. Pure Appl. Math.}}, 73(3):558--580, 2020.

\bibitem[Ilm93]{Ilmanen}
Tom Ilmanen.
\newblock Convergence of the {Allen}-{Cahn} equation to {Brakke}'s motion by mean curvature.
\newblock {\em J. Differ. Geom.}, 38(2):417--461, 1993.

\bibitem[Ilm94]{IlmanenER}
Tom Ilmanen.
\newblock {\em Elliptic regularization and partial regularity for motion by mean curvature}, volume 520 of {\em Mem. Am. Math. Soc.}
\newblock Providence, RI: American Mathematical Society (AMS), 1994.

\bibitem[IW24]{IW}
Tom Ilmanen and Brian White.
\newblock Fattening in mean curvature flow.
\newblock {\em \url{https://arxiv.org/abs/2406.18703}}, 2024.

\bibitem[KT20]{KimTonegawa}
Lami Kim and Yoshihiro Tonegawa.
\newblock Existence and regularity theorems of one-dimensional {Brakke} flows.
\newblock {\em Interfaces Free Bound.}, 22(4):505--550, 2020.

\bibitem[LZ24]{LeeZhao}
Tang-Kai Lee and Xinrui Zhao.
\newblock Closed mean curvature flows with asymptotically conical singularities.
\newblock {\em \url{https://arxiv.org/abs/2405.15577}}, 2024.

\bibitem[MM88]{MartinMattila}
Miguel~Angel Mart{\'{\i}}n and Pertti Mattila.
\newblock k-dimensional regularity classifications for s-fractals.
\newblock {\em Trans. Am. Math. Soc.}, 305(1):293--315, 1988.

\bibitem[Mod87]{Modica}
Luciano Modica.
\newblock The gradient theory of phase transitions and the minimal interface criterion.
\newblock {\em Arch. Ration. Mech. Anal.}, 98:123--142, 1987.

\bibitem[Son97a]{Soner1}
Halil~Mete Soner.
\newblock Ginzburg-{Landau} equation and motion by mean curvature. {I}: {Convergence}.
\newblock {\em J. Geom. Anal.}, 7(3):437--475, 1997.

\bibitem[Son97b]{Soner2}
Halil~Mete Soner.
\newblock Ginzburg-{Landau} equation and motion by mean curvature. {II}: {Development} of the initial interface.
\newblock {\em J. Geom. Anal.}, 7(3):477--491, 1997.

\bibitem[Ton03]{Tonegawa}
Yoshihiro Tonegawa.
\newblock Integrality of varifolds in the singular limit of reaction-diffusion equations.
\newblock {\em Hiroshima Math. J.}, 33(3):323--341, 2003.

\bibitem[Whi02]{WICM}
Brian White.
\newblock Evolution of curves and surfaces by mean curvature.
\newblock {\em \url{https://arxiv.org/abs/math/0212407}}, 2002.

\bibitem[{Whi}09]{White09}
Brian {White}.
\newblock {Currents and flat chains associated to varifolds, with an application to mean curvature flow}.
\newblock {\em {Duke Math. J.}}, 148(1):41--62, 2009.

\bibitem[Whi21]{white21}
Brian White.
\newblock Mean curvature flow with boundary.
\newblock {\em Ars Inven. Anal.}, 2021:34, 2021.
\newblock Id/No 4.

\end{thebibliography}
\end{document}